\documentclass[11pt,a4paper]{amsart}
\usepackage{amssymb,amsmath}

\usepackage{mathpazo}

\usepackage[bookmarks,
colorlinks,
pagebackref
]{hyperref}

\numberwithin{equation}{section}
\newtheorem{theorem}{Theorem}[section]

\newtheorem{lemma}[theorem]{Lemma}
\newtheorem{proposition}[theorem]{Proposition}

\theoremstyle{definition}

\newtheorem{remark}[theorem]{Remark}

\newcommand\Supp{\operatorname{Supp}}

\newcommand\Tor{\operatorname{Tor}}
\newcommand\Hom{\operatorname{Hom}}

\newcommand\Ext{\operatorname{Ext}}
\newcommand\Rad{\operatorname{Rad}}
\newcommand\Ker{\operatorname{Ker}}
\newcommand\Coker{\operatorname{Coker}}

\newcommand{\xx}{\underline x}

\newcommand{\lam}{\Lambda^I}

\newcommand{\Cech}{{\Check {C}}_{\xx}}

\title{A criterion for $I$-adic completeness}

\author{Peter Schenzel}

\address{Martin-Luther-Universit\"at Halle-Wittenberg,
Institut f\"ur Informatik, D --- 06 099 Halle (Saale),
Germany}

\email{peter.schenzel@informatik.uni-halle.de}



\begin{document}

\begin{abstract} 
Let $I$ denote an ideal in a commutative Noetherian ring $R$. Let $M$ be an 
$R$-module. The $I$-adic completion is defined by 
$\hat{M}^I = \varprojlim{}_{\alpha} M/I^{\alpha}M$. Then $M$ is called $I$-adic 
complete whenever the natural homomorphism $M \to \hat{M}^I$ is an isomorphism. 
Let $M$ be $I$-separated, i.e. $\cap_{\alpha} I^{\alpha}M = 0$. In the main 
result of the paper it is shown that $M$ is $I$-adic complete if and only if 
$\Ext_R^1(F,M) = 0$ for the flat test module $F = \oplus_{i = 1}^r R_{x_i}$ 
where $\{x_1,\ldots,x_r\}$ is a system of elements such that $\Rad I = \Rad \xx R$. 
This result extends several known statements starting with C. U. Jensen's result 
(see \cite[Proposition 3]{J})
that a finitely generated $R$-module $M$ over a local ring $R$ is complete if and only 
if $\Ext^1_R(F,M) = 0$ for any flat $R$-module $F$. 
\end{abstract}

\maketitle

\section{Introduction} Let $R$ denote a commutative Noetherian ring. For an ideal 
$I \subset R$ we consider the $I$-adic completion. For an $R$-module $M$ it is defined by 
$\varprojlim{}_{\alpha} M/I^{\alpha}M = \hat{M}^I$. In the case of $(R,\mathfrak{m})$ 
a local ring and a finitely generated $R$-module $M$ it was shown (see \cite{pS4}) that the following conditions 
are equivalent: 
\begin{itemize}
\item[(i)] $M$ is $\mathfrak{m}$-adic complete.
\item[(ii)] $\Ext_R^1(F,M) = 0$ for any flat $R$-module $F$.
\item[(iii)] $\Ext^1_R(F,M) = 0$ for $F = \oplus_{i=1}^r R_{x_i},$ where $x_1,\ldots,x_r \in 
\mathfrak{m}$ are elements that generate an $\mathfrak{m}$-primary ideal. 
\end{itemize}
Here $R_x$ denotes 
the localization of $R$ with respect to $\{x^{\alpha} | \alpha \in \mathbb{N}_{\geq0} \}$.
This is an extension (in the local case) of a result of C. U. Jensen (see \cite[Proposition 3]{J}) who 
proved the  equivalence of the first two conditions. The main result of the present 
paper is an extension to the case of $I$-adic completion. More precisely we prove the following 
result:

\begin{theorem} \label{1.1} Let $I$ be an ideal of a commutative Noetherian ring $R$. 
Let $M$ denote an arbitrary $R$-module that is $I$-separated, i.e. $\cap_{\alpha} I^{\alpha}M = 0$. 
Then  the following conditions are equivalent: 
\begin{itemize}
\item[(i)] $M$ is $I$-adic complete.
\item[(ii)] $\Ext_R^i(F,M) = 0$ for all $i \geq 0$ and any flat $R$-module $F$ with 
$F \otimes_R R/I = 0$.
\item[(iii)] $\Ext_R^1(R_x,M) = 0$ for all elements $x \in I$.
\item[(iv)] $\Ext_R^1(\oplus_{i=0}^r R_{x_i},M) = 0$ for a system of elements $\underline{x} = \{x_1,\ldots,x_r\}$ such that $\Rad \underline{x}R = \Rad I$.
\end{itemize}
\end{theorem}

Note that  in his paper (see \cite[Proposition 3]{J}) 
C. U. Jensen proved the following: Let $R$ denote a semi local ring and $M$ a finitely generated $R$-module. 
Then $M$ is complete if and only if $\Ext_R^1(F,M) = 0$ 
for all flat (resp. countably generated flat) $R$-modules $F$. A different proof 
for the vanishing of $\Ext_R^i(F,M)$ for all $i\geq 1$,  all flat $R$-modules and $M$ a complete 
$R$-module follows by results of R.-O. Buchweitz and H. Flenner (see \cite[Theorem 2]{BF}).

Moreover another criterion for $I$-adic completeness was shown by 
A. Frankild and S. Sather-Wagstaff. Let $I \subset R$ an ideal contained in the Jacobson radical of $R$. 
Let $M$ be a finitely generated $R$-module. In their paper (see \cite{FSW}) they proved that 
$M$ is $I$-adically complete if and only if $\Ext^i_R(\hat{R}^I, M) = 0$ for 
all $1 \leq i \leq \dim_R M$, where $\hat{R}^I$ denotes the $I$-adic completion 
of $R$.

Whence, the main result of the present paper is the construction of a simple flat test module $F$ for the 
completeness of $M$ in the $I$-adic topology in terms of the vanishing of $\Ext_R^1(F,M)$. This generalizes 
the case of the maximal ideal proved in \cite{pS4}. Instead of Matlis Duality used in \cite{pS4} (not 
available in the case of $I$-adic topology) here we use some homological techniques.

\section{Preliminary Results} 
In the following $R$ denotes  a commutative Noetherian ring.  For an ideal 
$I \subset R$ we consider the $I$-adic completion. For an $R$-module $M$ it is defined by 
$\varprojlim{}_{\alpha} M/I^{\alpha}M = \hat{M}^I$. We say that $M$ is $I$-adically 
complete whenever the natural homomorphism $\tau = \tau_M^I : M \to \hat{M}^I$ (induced 
by the natural surjections $M \to M/I^{\alpha}M$) is an isomorphism. It 
implies that $M$ is $I$-separated, i.e. $\cap_{\alpha} I^{\alpha} M = 0$. We begin with a 
few preparatory results needed later on. 

\begin{proposition} \label{2.1} Let $I \subset R$ denote an ideal. Let $M$ denote 
an  $R$-module such that $M \otimes_R R/I = 0$. Then $M \otimes_R X = 0$ for any 
$R$-module $X$ with  $\Supp_R X \subseteq V(I)$.  
\end{proposition}

\begin{proof} The assumption $M \otimes_R R/I = 0$ implies that $M = I^{\alpha} M$ 
for all $\alpha \geq 1$, or in other words $M \otimes_R R/I^{\alpha} = 0$. First of all let 
$X$ be a finitely generated $R$-module. Then $I^{\alpha}X = 0$ for a certain $\alpha \in \mathbb{N}$.  
Therefore 
\[
M \otimes_R X = M \otimes_R X/I^{\alpha}X = (M \otimes_R R/I^{\alpha})\otimes_R X = 0.
\]
If $X$ is not necessarily a finitely generated $R$-module, then $X \simeq \varinjlim X_{\alpha}$ 
for a direct system of finitely generated submodules $X_{\alpha} \subset X$ ordered by inclusions. 
Whence $\Supp_R X_{\alpha} \subseteq V(I)$ and 
\[
M \otimes_R X = \varinjlim (M \otimes_R X_{\alpha}) = 0,
\] 
as required.
\end{proof} 

Let $I \subset R$ denote an ideal. For an $R$-module $X$ the left derived 
functors $\lam_i(X), i \in \mathbb{Z}, $ are defined by 
$H_i(\varinjlim{}_{\alpha} (R/I^{\alpha} \otimes_R F_{\cdot}))$, where $F_{\cdot}$ 
denotes a flat resolution of $X$. The functors $\lam_i(\cdot)$ were first 
systematically studied by Greenlees and May (see \cite{GM}) and by A.-M. Simon 
(see \cite{amS}) and more recently in \cite{ALL} and \cite{pS2}. Note that there 
is a natural surjective homomorphism $\lam_0(X) \to \hat{X}^I$. Since $R$ is Noetherian  
it follows that $M/I^{\alpha}M \simeq \hat{M}^I/I^{\alpha}\hat{M}^I$ 
for all $\alpha \geq 1$ for any $R$-module $M$ (see \cite[Section 2.2]{jS}). 
Whence the completion $\hat{M}^I$ is $I$-adically complete.

In the following we will discuss the assumption on the $R$-module $M$ in 
Proposition \ref{2.1}. It provides a certain kind of a Nakayama Lemma. Its proof 
is known (see \cite[5.1]{amS} and also \cite[1.3, Lemma (ii)]{amS2}).  Here we add 
these arguments. 

\begin{proposition} \label{2.2} Let $I$ denote an ideal in a commutative Noetherian 
ring $R$. Let $M$ denote an arbitrary $R$-module. Then $M \otimes_R R/I = 0$ if and only 
if $\hat{M}^I = 0$ and if and only if $\lam_0(M) = 0$.
\end{proposition} 

\begin{proof} If $M \otimes_R R/I = 0$, then $M/I^{\alpha}M = 0$ for all $\alpha \geq 1$
and $\hat{M}^I = \varprojlim{}_{\alpha} M/I^{\alpha}M = 0$. Conversely, suppose that 
$\hat{M}^I = 0$. Then $0 = \hat{M}^I/I\hat{M}^I \simeq M/IM$, as required. If $\lam_0(M) = 0$, 
then clearly $\hat{M}^I = 0$. Let $F_1 \to F_0 \to M \to 0$ denote a free presentation 
of $M$. If $\lam_0(M) = 0$, then 
$\hat{F}_1^I \to \hat{F}_0^I$ is onto. Therefore 
$F_1/IF_1 \to F_0/IF_0$ is onto too, i.e. $M/IM = 0$ as required.  
\end{proof} 

In the following let $\xx = \{x_1,\ldots,x_r\}$ denote a system of elements of $R$ 
such that $\Rad \xx R = \Rad I$. Then we consider the \v{C}ech complex $\Cech$ 
of the system $\xx$ (see e.g. \cite[Section 3]{pS2}). Note that this is the same 
as the complex $K_{\infty}(\underline{x})$  introduced in \cite{ALL} as the 
direct limit of the Koszul complexes $K_{\cdot}(\underline{x}^{\underline{n}})$. 
There is a natural morphism $\Cech \to R$. 
 
In our context the importance of the \v{C}ech complex  
is the following result. It allows the expression of the right derived 
functors $\lam_i(\cdot)$ in different terms.

\begin{theorem} \label{2.3} Let $I \subset R$ denote an ideal of a commutative Noetherian 
ring $R$. Let $M$ denote an $R$-module. Then there are natural isomorphisms 
\[
\lam_i(M) \simeq 
H_i(\Hom_R(\Cech,E_R(M)^{\cdot})) 
\]
for all $i \in \mathbb{Z}$, where $E_R(M)^{\cdot}$ denotes an injective resolution of $M$.
\end{theorem} 

\begin{proof} This statement is one of the main results of \cite[Section 4]{pS2}. In the 
formulation here we do not use the technique of derived functors and derived categories. 
(For a more advanced exposition based on derived categories and derived functors the 
interested reader might also consult \cite{ALL}.)
\end{proof} 

Let us summarize a few basic results on completions used in the sequel. The following results are 
well-known. 

\begin{proposition} \label{2.5} Let $J \subset I$ denote ideals of a commutative Noetherian ring.
Let $\xx = \{x_1,\ldots,x_r\}$ denote a system of elements of $R$. Let $M$ denote an $R$-module. 
\begin{itemize}
\item[(a)] Suppose that $M$ is $I$-adic complete. Then it is also $J$-adic complete. 
\item[(b)] Suppose that $\Rad I = \Rad \xx R$. Assume  that $M$  is complete in the $x_iR$-adic topology 
for $i = 1,\ldots,r$ and $M$ is $I$-separated. Then $M$ is $I$-adic complete.
\end{itemize} 
\end{proposition} 

\begin{proof} For the proof of (a) see \cite[2.2.6]{jS}. 
 Let $\hat{M}^I$ denote the $I$-adic completion of $M$.  Because the $I$-adic topology and the topology defined by $\{(x_1^{\alpha},\ldots,x_r^{\alpha})R\}_{\alpha \geq 0}$ are equivalent it follows 
$\hat{M}^I \simeq \varprojlim{}_{\alpha}M/(x_1^{\alpha},\ldots,x_r^{\alpha})M$. Therefore an element $m \in 
\hat{M}^I$ has the form $m = \sum_{\alpha \geq 0} \sum_{i=1}^r x_i^{\alpha} m_{\alpha,i}$ with 
$m_{\alpha,i} \in M$. Since $M$ is $x_iR$-adically complete $M \simeq \varprojlim{}_{\alpha} 
M/x_i^{\alpha}M$ and $m_i = \sum_{\alpha} x^{\alpha}_i m_{\alpha,i}$ for some $m_i \in M$. Then 
$m' = \sum_{i=1}^r m_i \in M$ maps to $m$ under $\tau : M \to \hat{M}^I$. Because $M$ is $I$-separated it is $I$-adically complete. 
This proves (b).
\end{proof} 

The following result is helpful in order to compute projective limits of $\Ext$-modules. 
Here $\varprojlim{}^1$ denotes the first right derived functor of the projective limit 
(see \cite{cW}).

Let $M$ denote an $R$-module and $x \in R$. We consider the following projective system 
$\{M,x\}$, where $M_{\alpha} = M$ for all $\alpha \in \mathbb{N}$ and the transition map 
$M_{\alpha +1} \stackrel{x}{\to} M_{\alpha}$ is the multiplication by $x$. 

\begin{lemma} \label{2.6} With the previous notation it follows that 
\[
\varprojlim{}^i \{M,x\} \simeq \Ext^i_R(R_x,M) \text{ for } i = 0,1 \text{ and } \Ext_R^i(R_x,M) = 0 \text{ for all } i \geq 2.
\]
If $M$ is $xR$-separated, then $\varprojlim \{M,x\} =  \Hom_R(R_x, M) = 0$.
\end{lemma} 

\begin{proof} Let $\{R,x\}$ denote the direct system with $R_{\alpha} = R$ for all $\alpha \in \mathbb{N}$ and the transition map 
$R_{\alpha} \stackrel{x}{\to} R_{\alpha + 1}$ is the multiplication by $x$. Then $R_x \simeq \varinjlim{}_{\alpha} \{R,x\}$ as it is well 
known. Therefore, $\Hom_R(\{R,x\},M) \simeq \{M,x\}$. Now, by \cite[Lemma 2.6]{pS3} (see also Lemma \ref{4.1})  there are short 
exact sequences 
\[
0 \to \varprojlim{}^1 \{\Ext^{i-1}_R(R, M), x\} \to \Ext_R^i(R_x, M) \to \varprojlim \{\Ext_R^i(R,M),x\}
\to 0
\]
for all $i \in \mathbb{Z}$. Whence the results of the first part follow. 
If $M$ is $xR$-separated, the vanishing of $\Hom_R(R_x,M)$ is easy to see.
\end{proof}

The previous result is a slight modification of Lemma \cite[Lemma 2.7]{pS4}. 

\begin{remark} \label{2.7} A morphism of two complexes $X^{\cdot} \to Y^{\cdot}$ of $R$-modules is called 
a quasi-isomor\-phism (or homology isomorphism) whenever the induced homomorphisms on the cohomo\-logy modules 
$H^i(X^{\cdot}) \to H^i(Y^{\cdot})$ 
are isomorphisms for all $i \in \mathbb{Z}$. For the definitions of $\Hom_R(X^{\cdot}, Y^{\cdot})$ and 
$X^{\cdot} \otimes_R Y^{\cdot}$ we refer to \cite{cW} and \cite{AF}. For the details on homological algebra of 
complexes of modules we refer to \cite{AF}.
\end{remark}

\section{On \texorpdfstring{$I$}--adic Completions}
Let $I \subset R$ denote an ideal of a commutative Noetherian ring $R$. 
In the first main result we shall prove a vanishing result for a 
certain $\Ext$-module for an $I$-adic complete $R$-module.

\begin{theorem} \label{3.1} Let $M$ denote an arbitrary $R$-module. 
Let $F$ denote a flat $R$-module satisfying $F \otimes_R R/I = 0$. 
Then $\Ext^i_R(F,\hat{M}^I) = 0$ for all $i \in \mathbb{Z}$.  
\end{theorem} 

\begin{proof} For the proof we use the techniques summarized in Theorem \ref{2.3}. To this end we fix 
a few notation. Let $\xx = \{x_1,\ldots,x_r\}$ denote a system of elements of $R$ such that $\Rad \xx R 
= \Rad I$. Let $\Cech$ denote the \v{C}ech complex with respect to $\xx$ as defined in  \cite[Section 3]{pS2}. 
It is a bounded complex of flat $R$-modules with a natural morphism $\Cech \to R$ 
of complexes. 
Let $E^{\cdot}$ denote an injective resolution of $\hat{M}^I$ as an $R$-module.  
By applying $\Hom_R(\cdot,E^{\cdot})$ it induces a morphism of complexes $E^{\cdot} \to 
\Hom_R(\Cech, E^{\cdot})$. 

Next we investigate the complex 
$\Hom_R(\Cech,E^{\cdot})$. It is a left bounded complex of injective $R$-modules. Moreover, 
by virtue of Theorem \ref{2.3} it follows that $\lam_i(\hat{M}^I) \simeq H_i(\Hom_R(\Cech,E^{\cdot}))$ 
for all $i \in \mathbb{Z}$. 
Therefore $ H_i(\Hom_R(\Cech,E^{\cdot})) = 0$ for all $i \not= 0$ and $ H_0(\Hom_R(\Cech,E^{\cdot})) 
\simeq \hat{M}^I$ since $\hat{M}^I$ is $I$-adic complete and $\lam_i(\hat{M}^I) = 0$ for all $i > 0$ 
(see \cite[5.2]{amS}). Therefore, the morphism 
$E^{\cdot} \to \Hom_R(\Cech, E^{\cdot})$ is a quasi-isomorphism.

Now we consider the complexes
\[
\Hom_R(F, \Hom_R(\Cech,E^{\cdot})) \simeq \Hom_R(F \otimes_R \Cech,E^{\cdot}).
\]
We claim that they are  homologically trivial complexes. For that reason it will be enough to show that $F \otimes_R \Cech$ 
is homologically trivial since $E^{\cdot}$ is a bounded below complex of injective $R$-modules. 
Since $F$ is a flat $R$-module it yields 
that $H^i(F \otimes_R \Cech) \simeq F \otimes_R H^i(\Cech) =0$ for all $i \in \mathbb{Z}$. For 
the vanishing apply Proposition \ref{2.1} with $H^i(\Cech) \simeq 
H^i_I(R)$ for all $i \in \mathbb{Z}$ and $\Supp_R H^i_I(R) \subseteq V(I)$. 

As shown above  $E^{\cdot} \to \Hom_R(\Cech, E^{\cdot})$ is a quasi-isomorphism of left bounded complexes of 
injective $R$-modules. Applying the functor $\Hom_R(F, \cdot)$ induces a quasi-isomorphism 
\[
\Hom_R(F,E^{\cdot}) \to \Hom_R(F, \Hom_R(\Cech, E^{\cdot})).
\]
The complex  at the right is cohomologically trivial, while the cohomology of the complex at the left
is $\Ext_R^i(F,\hat{M}^I)$. Therefore, the $\Ext$-modules vanish, as required.
\end{proof}

In the following result we prove another behaviour of certain $\Ext$-modules in respect to the $I$-adic 
completion. 

\begin{theorem}  \label{3.3} Let $I \subset R$ denote an ideal. Let $M, X$ be 
arbitrary $R$-modules with $\tau : M \to \hat{M}^I$ the natural map.   Suppose that $\Supp_R X \subseteq V(I)$. Then 
\begin{itemize}
\item[(a)] $\Ext^i_R(X,\hat{M}^I/\tau(M)) = 0$ and 
\item[(b)]  the natural homomorphism $\Ext_R^i(X,M) \to \Ext_R^i(X, \hat{M}^I) $
is an isomorphism
\end{itemize}
 for all $i \in \mathbb{N}$ .
\end{theorem}

\begin{proof} Let $\xx = \{x_1,\ldots,x_r\}$ denote elements of $R$ such that $\Rad \xx R = \Rad I$. 
Let $\Cech$ denote the \v{C}ech complex with respect to $\xx$. Then there is a short exact sequence of 
complexes 
\[
0 \to D_{\xx}[-1] \to \Cech \to R \to 0,
\] 
where $D_{\xx}$ is the global \v{C}ech complex. 
That is $D_{\xx}^i = \Cech^{i+1}$ for $i \geq 0$ and $D_{\xx}^i = 0$ for $i < 0$. Now let 
$E^{\cdot}$ denote an injective resolution of $M$. By applying the functor $\Hom_R(\cdot, E^{\cdot})$ to the short 
exact sequence of complexes it provides a short exact sequence
\[
0 \to E^{\cdot} \to \Hom_R(\Cech,  E^{\cdot}) \to \Hom_R(D_{\xx}[-1],  E^{\cdot}) \to 0
\]
of left bounded complexes of injective $R$-modules. The complex in the middle is 
an injective resolution of $\hat{M}^I$ as follows by  \ref{2.3} and the fact that $\hat{M}^I$ is $I$-adic complete 
(see \cite[5.2]{amS}). Therefore the complex at the right is quasi-isomorphic to $\hat{M}^I/\tau(M)$ 
considered as a complex concentrated in homological degree zero. 
In order to prove the statement in (a) let $L_{\cdot}$ denote a projective resolution of $X$. Then 
$\Ext^i_R(X,\hat{M}^I/\tau(M)) \simeq H^i( \Hom_R(L_{\cdot}, \Hom_R(D_{\xx}[-1],E^{\cdot})))$ (see \cite{AF}) and 
it will be enough to show that the last modules vanish for all $i \in \mathbb{Z}$. 

To this end consider the isomorphism of complexes 
\[
\Hom_R(L_{\cdot}, \Hom_R(D_{\xx}[-1],E^{\cdot})) \simeq \Hom_R(L_{\cdot} \otimes_R D_{\xx}, E^{\cdot})[1].
\]
Now $D_{\underline{x}}$ is a bounded complex of flat $R$-modules. Therefore  there is the quasi-isomorphism $L_{\cdot} \otimes _R D_{\xx} \to X \otimes_R D_{\xx}$. It will be enough 
to show that the complex $X \otimes_R D_{\xx}$ is homologically trivial. But this is true since $X \otimes_R D_{\xx}^i = 0$ 
for all $i \in \mathbb{Z}$ because of $\Supp_R X \subseteq V(I)$. This proves the statement in (a). 

By view of the above  investigations there is a quasi-isomorphism $E^{\cdot} \to \Hom_R(\Cech, E^{\cdot})$ where 
the second complex is an injective resolution of $\hat{M}^I$. This proves the statement in (b).
\end{proof}

We will continue with a result on the vanishing of certain $\Ext$-modules. 
As a step towards to our main result we shall prove a partial result in order to 
characterize the completion. 

\begin{theorem} \label{3.4} Let $R$ denote a commutative Noetherian  ring. Let $M$ denote an 
arbitrary $R$-module. Let $x \in R$ be an element such that $0:_M x^{\alpha} = 0 :_M x^{\beta}$ 
for all $\alpha \geq \beta$. Then  the following conditions are equivalent:
\begin{itemize}
\item[(i)] $M$ is $xR$-adic complete.
\item[(ii)] $\Ext_R^i(R_x,M) = 0$  for $i = 0,1$.
\end{itemize}
\end{theorem}

\begin{proof} Let $N = \cup_{\alpha \geq 1} 0 :_M x^{\alpha}$. Then $N = 0:_M x^{\beta}$ as 
follows by the definition of $\beta$. 
That is,  in both of the statements we may replace 
$x$ by $x^{\beta}$ without loss of generality. That is,  we may assume that $x N = 0$. Then there is a 
commutative diagram  with exact rows:
\[
\begin{array}{ccccccl}
0   \to & M/N & \stackrel{x^{\alpha +1}}{\longrightarrow} & M & \to & M/x^{\alpha +1}M &  \to 0 \\
           &  \downarrow x &                                          & \parallel &  & \downarrow              &               \\
0   \to & M/N & \stackrel{x^{\alpha}}{\longrightarrow} &  M  & \to  & M/x^{\alpha} M& \to   0. 
\end{array}
\] 
By passing to the inverse limit it provides an exact sequence 
\[
0 \to \varprojlim \{M/N, x\} \to M \to \hat{M}^x \to \varprojlim{}^1 \{M/N, x\} \to 0.
\]
By view of Lemma \ref{2.5} it yields that $\varprojlim{}^{i} \{M/N, x\} \simeq \Ext_R^i(R_x,M/N)$ for 
$i = 0,1$. 
Furthermore, the short exact sequence 
$0 \to N \to M \to M/N \to 0$ induces an isomorphism $\Ext_R^i(R_x,M) \simeq \Ext_R^i(R_x,M/N)$. 
This follows by the long exact cohomology sequence and $\Ext_R^i(R_x,N) = 0$ for all 
$i \in \mathbb{Z}$. For this vanishing note that the multiplication by $x$ acts on $R_x$ as an 
isomorphism and on $N$ as the zero map. 

That is,  the  homomorphism $M \to \hat{M}^x$ is an isomorphism, if and only if  
$\Ext^i_R(R_x,M) = 0$ for $i = 0,1$ which proves (i) $\Longleftrightarrow$ (ii). 
\end{proof}

\noindent {\it Proof of Theorem \ref{1.1}.}  The implication (i) $\Longrightarrow$ (ii) is a 
consequence of Theorem \ref{3.1}. Moreover, the implications  (ii) $\Longrightarrow$ (iii) 
is trivial and (iii) $\Longrightarrow$ (iv)
is easy to see. In order to prove (iv) $\Longrightarrow$ (i) we first note that 
$\Ext_R^1(R_{x_i}, R) = 0$ 
for all $i = 1,\ldots,r$. By Theorem \ref{3.4} this implies that $M$ is $x_iR$-adically 
complete for $i = 1,\ldots,r$. Since $\underline{x} = \{x_1,\ldots,x_r\}$ generates $I$ up to the radical 
it follows that $M$ is $I$-adic complete by Proposition \ref{2.5} since $M$ is $I$-adic separated. This 
completes the proof. 
\hfill $\Box$

\section{An alternative proof of Theorem \ref{3.1}}
In this section there is an alternative proof of Theorem \ref{3.1} based on some results 
on inverse limits. The results of the following two lemmas might be also of some independent 
interest. These statements are particular cases of certain 
spectral sequences on inverse limits (see \cite{J3}). Here we give an elementary proof based 
on the description of $\varprojlim{}^1$ as discussed in \cite{cW}. 

\begin{lemma} \label{4.1} Let $R$ denote a commutative ring. 
 Let $\{M_{\alpha}\}_{\alpha \in \mathbb{N}}$ be a direct system of $R$-modules. Let $N$ denote an arbitrary $R$-module.
Then there is a short exact sequence
\[
0 \to \varprojlim{}^1 \Ext^{i-1}_R(M_{\alpha}, N) \to \Ext_R^i(\varinjlim M_{\alpha}, N) \to \varprojlim \Ext_R^i(M_{\alpha},N)
\to 0
\]
for all $i \in \mathbb{Z}$. 
\end{lemma}

\begin{proof} This result was proved in (\cite[Lemma 2.6]{pS3}) for a Noetherian ring. The same 
arguments work in the general case. \end{proof} 

In the following we need a certain dual statement of Lemma \ref{4.1}. In a certain sense it will be 
the key argument for the second proof of Theorem \ref{3.1}. 

\begin{lemma} \label{4.2} Let $R$ denote a commutative  ring. Let $M$ denote an 
arbitrary $R$-module. Let $\{N_{\alpha}\}_{\alpha \in \mathbb{N}}$ be an inverse  system of $R$-modules with 
$\varprojlim{}^1N_{\alpha} = 0$. 
Then there is a short exact sequence
\[
0 \to \varprojlim{}^1 \Ext^{i-1}_R(M,N_{\alpha}) \to \Ext_R^i(M,\varprojlim N_{\alpha}) \to \varprojlim 
\Ext_R^i(M,N_{\alpha})\to 0
\]
for all $i \in \mathbb{Z}$. 
\end{lemma} 

\begin{proof} Because of $\varprojlim{}^1N_{\alpha} = 0$ there is a short exact sequence 
\[
0 \to \varprojlim N_{\alpha} \to \prod N_{\alpha} \to \prod N_{\alpha} \to 0,
\]
where the third homomorphism is the transition map (see \cite{cW}). It induces a long exact cohomology sequence 
\begin{gather*}
\cdots \to \prod \Ext_R^{i-1}(M, N_{\alpha} ) \stackrel{f}{\to} \prod \Ext^{i-1}_R(M, N_{\alpha} ) \to 
\Ext^i_R(M,\varprojlim N_{\alpha}) \\
\to \prod \Ext^i_R(M, N_{\alpha} ) \stackrel{g}{\to} \prod \Ext^i_R(M, N_{\alpha} ) \to \cdots .
\end{gather*} 
To this end recall that $\Ext$ transforms direct products into direct products in the second variable 
and cohomology commutes with direct products 
(see e. g. \cite{EJ}). Now it is known (see e.g. \cite{cW})  that 
\[
\Coker f \simeq \varprojlim{}^1 \Ext^{i-1}_R(M,N_{\alpha}) \text{ and } 
\Ker g = \varprojlim \Ext_R^i(N,  N_{\alpha}).
\]
This completes the proof.
\end{proof}

\begin{remark} The assumption that $\varprojlim{}^1N_{\alpha} = 0$ is fulfilled whenever the 
projective system $\{N_{\alpha}\}$ satisfies the Mittag-Leffler condition. That is, for instance 
when the transition map $N_{\alpha + 1} \to N_{\alpha}$ is surjective for all $\alpha \geq 1$. 
\end{remark} 

Note that the proof of the following Theorem is motivated by some arguments done 
by Buchweitz and Flenner (see \cite{BF}). 

\begin{theorem} \label{4.3} Let $R$ denote a commutative ring. Let $M$ denote an  $R$-module. 
Let $F$ denote a flat $R$-module satisfying $F \otimes_R R/I = 0$. 
Then $\Ext^i_R(F,\hat{M}^I) = 0$ for all $i \in \mathbb{Z}$.  
\end{theorem} 

\begin{proof}  By definition we have $\hat{M}^I = \varprojlim{}_{\alpha} M/I^{\alpha}M$. 
By Lemma \ref{4.2} there is a short exact sequence 
\[
0 \to \varprojlim{}^1 \Ext^{i-1}_R(F,M/I^{\alpha}M) \to \Ext_R^i(F,\hat{M}^I) 
\to \varprojlim \Ext^i_R(F,M/I^{\alpha}M) \to 0.
\]
In order to show the vanishing of $\Ext_R^i(F,\hat{M}^I)$ for all $i \in \mathbb{Z}$ it will be enough to 
show the vanishing of $\Ext^i(F,M/I^{\alpha}M)$ for all $i \in \mathbb{Z}$ and all $\alpha \geq 1$. 
We claim that 
\[
\Ext_R^i(F,M/I^{\alpha}M) \simeq \Ext^i_{R/I^{\alpha}}(F/I^{\alpha}F,M/I^{\alpha}M) 
\]
for all $i \in \mathbb{Z}$ and all $\alpha \geq 1$. 
In order to show these isomorphisms let $L_{\cdot}$ be a projective 
resolution of $F$ as an $R$-module. As $F$ is flat as an $R$-module $\Tor_i^R(F,R/I^{\alpha}) =0$ 
for all $i > 0$ and $L_{\cdot} \otimes_R R/I^{\alpha}$ is a projective resolution of $F/I^{\alpha}F$ 
as an $R/I^{\alpha}$-module. By adjunction there are isomorphisms of complexes 
\[
\Hom_{R/I^{\alpha}}(L_{\cdot} \otimes_R R/I^{\alpha}, M/I^{\alpha}M) \simeq \Hom_R(L_{\cdot}, M/I^{\alpha}M).
\]
By taking cohomology it proves the above claim. 
The vanishing follows now  because of $F/I^{\alpha} F =0$ as a consequence of the assumption $F\otimes_R R/I = 0$. 
\end{proof} 
 
\begin{remark} \label{4.4} In fact Theorem \ref{4.3} is a slight sharpening of Theorem \ref{3.1}. To this end 
recall that for an $R$-module $M$ and an ideal $I \subset R$ the $I$-adic completion $\hat{M}^I$ is not necessarily 
$I$-adic complete (as used in the proof of Theorem \ref{3.1}). For an explicit example see J. Bartijn's Thesis 
\cite[I, \S 3, page 19]{jB}. Note that in this example the ring is not Noetherian. It grows out of 
\cite[III, \S 2, Exerc. 12]{nB}. See also  A.~Yekutieli (see \cite[Example 1.8]{aY}).  
\end{remark}

\section*{Acknowledgements} The author thanks Anne-Marie Simon and the reviewer for a careful reading of the manuscript and suggesting some comments.

\end{document}